\documentclass[12pt,a4paper]{article}

\usepackage{amsmath, amsthm, amsfonts, amssymb,color,geometry,enumitem}
\usepackage{graphicx}
\usepackage{fancybox}
\usepackage{cases}
\usepackage{fancyhdr}
\usepackage{graphicx}
\usepackage{overpic}

\theoremstyle{definition}

\newtheorem{theorem}{Theorem}[section]
\newtheorem{proposition}[theorem]{Proposition}
\newtheorem{lemma}[theorem]{Lemma}
\newtheorem{definition}[theorem]{Definition}
\newtheorem{remark}[theorem]{Remark}

\makeatletter

\@addtoreset{equation}{section}
\makeatother

\newcommand{\R}{\mathbb{R}}   
\newcommand{\C}{\mathbb{C}}   
\newcommand{\N}{\mathbb{N}}   
\newcommand{\im}{\text{\normalfont Im}}    
\newcommand{\re}{\text{\normalfont Re}}    
\renewcommand{\epsilon}{\varepsilon}    

\newcommand{\dd}{\mathrm{d}} 
\newcommand{\Restr}[2]{{#1}{\restriction}_{#2}} 
\newcommand\FN{\widetilde{F}^{[N]}}  

\begin{document}

\date{}

\title{On the free L\'evy measure of the normal distribution}
\author{Takahiro Hasebe and Yuki Ueda}
\maketitle

\begin{abstract} Belinschi et al.\ \cite{BBLS11} proved that the normal distribution is freely infinitely divisible. 
This paper establishes a certain monotonicity, real analyticity and asymptotic behavior of the density of the free L\'evy measure. The monotonicity property strengthens the result in Hasebe et al.\ \cite{HST19} that the normal distribution is freely selfdecomposable. 
\end{abstract}

\section{Introduction}\label{sec:intro}
\subsection{Backgrounds}
The role of the normal distribution is played by Wigner's semicircle distribution in free probability. Most notably, the latter appears in the free central limit theorem (see e.g.\ \cite{HP00,NS06,VDN92}). Although a role of the normal distribution in free probability is not very obvious, there are still some attempts to understand it.  In \cite{BBLS11} the normal distribution was proven to be freely infinitely divisible, and then, as a stronger result, the normal distribution was proven to be freely selfdecomposable in \cite{HST19}. Combinatorial aspects are also investigated in \cite{BBLS11}.  

This paper further analyzes the free infinite divisibility of the normal distribution. Key analytical machineries are the Cauchy transform, its reciprocals and the Voiculescu transform, defined as follows. The \emph{Cauchy transform} of a probability measure $\mu$ on $\R$ is the function
\[
G_\mu(z)=\int_\R \frac{1}{z-x}\mu(\dd x), \qquad z\in \C\setminus \R.
\]
It is easy to see that $G_\mu$ is analytic and maps the complex upper half-plane (denoted $\C^+$) to the lower half-plane (denoted $\C^-$) and also $\C^-$ to $\C^+$. Note that the Cauchy transform is often defined only on $\C^+$ but in this paper the values on $\C^-$ is also useful, see \eqref{eq:CauchyND}. 
Then the \emph{reciprocal Cauchy transform} of $\mu$ is defined to be 
\begin{equation} \label{eq:F}
F_\mu(z)=\frac{1}{G_\mu (z)}, \qquad z\in \C^+, 
\end{equation}
which is an analytic selfmap of $\C^+$. 

According to \cite[Proposition 5.4]{BV93}, for a probability measure $\mu$ on $\R$ and $\lambda>0$ there exist positive numbers $\alpha,\beta$ and $M$ such that $F_\mu$ is univalent on the set $\Gamma_{\alpha,\beta}:=\{z\in \C^+: \im (z)>\beta, |\re(z)|<\alpha \im(z)\}$ and $\Gamma_{\lambda,M}\subset F_\mu(\Gamma_{\alpha,\beta})$. Therefore, the compositional right inverse $F_\mu^{-1}$ is defined on $\Gamma_{\lambda,M}$. The {\it Voiculescu transform} $\varphi_\mu$ is defined by
\[
\varphi_\mu(z):=F_\mu^{-1}(z)-z,\qquad z\in \Gamma_{\lambda,M}.
\]
For probability measures $\mu$ and $\nu$ on $\R$, the {\it free additive convolution} $\mu\boxplus\nu$ is a unique probability measure satisfying that
\[
\varphi_{\mu\boxplus\nu}(z)=\varphi_\mu(z)+\varphi_\nu(z)
\]
on the intersection of the domains where three Voiculescu transforms are defined.

A probability measure $\mu$ on $\R$ is said to be {\it freely infinitely divisible} if for any $n\in \N$ there exists a probability measure $\mu_n$ on $\R$ such that
\[
\mu=\underbrace{\mu_n\boxplus \cdots \boxplus \mu_n}_{n \text{ times}}.
\]
A basic fact is a characterization of freely infinitely divisible distributions  in terms  of the Voiculescu transform. 
\begin{theorem}[{\cite[Theorem 5.10]{BV93}}]
A probability measure $\mu$ on $\R$ is freely infinitely divisible if and only if the Voiculescu transform $\varphi_\mu$ has an analytic extension defined on $\C^+$ with values in $\C^-\cup \R$.
\end{theorem}
For a freely infinitely divisible distribution $\mu$, let $\widetilde{\varphi}_\mu$ denote the analytic extension of its Voiculescu transform as described above. Then the transform $\widetilde{\varphi}_\mu$ has the following Pick--Nevanlinna representation:
\[
\widetilde{\varphi}_\mu(z)=b_\mu+\int_\R \frac{1+xz}{z-x}\tau_\mu(\dd x), \qquad z\in \C^+,
\]
for some $b_\mu\in \R$ and a finite measure $\tau_\mu$ on $\R$. The pair $(b_\mu, \tau_\mu)$ is unique. The measure 
\begin{equation}
\nu_\mu(\dd x) = \frac{1+x^2}{x^2}\mathbf1_{\R\setminus\{0\}}(x) \,\tau_\mu(\dd x)   \label{eq:freeLM}
\end{equation}
is called the \emph{free L\'evy measure} of $\mu$ and the mass of $\tau_\mu$ at zero is called the \emph{semicircular component}. The standard semicircle distribution (i.e.\ with mean 0 and variance 1) corresponds to $(b_\mu,\tau_\mu) = (0,\delta_0)$. 

For many classical distributions including the normal distribution,  its Voiculescu transform cannot be explicitly calculated. In such a case, the following condition has been a useful sufficient (but not necessary, see \cite[Proposition 3.6]{AH13b} for a counterexample) condition for proving the free infinite divisibility. 
\begin{definition}[{\cite[Definition 5.1]{AH13}}]  \label{def:UI}
A probability measure $\mu$ on $\R$ is said to be in class ${\bf UI}$ (denoted by $\mu\in {\bf UI}$) if $F_\mu^{-1}$, defined in some $\Gamma_{\lambda,M}$, analytically extends to a univalent map in $\C^+$, or  equivalently, if there exists a domain $ \C^+ \subset \Omega \subset \C$ such that $F_\mu$ extends to an analytic bijection from $\Omega$ onto $\C^+$.
\end{definition}

\begin{proposition}[{\cite[Proposition 5.2]{AH13}}]
If $\mu \in {\bf UI}$, then $\mu$ is freely infinitely divisible.
\end{proposition}

Many classical distributions, despite their unexplicit Voiculescu transforms,  have been proven to be in class {\bf UI}. They include the normal distribution \cite{BBLS11}, some beta distributions and some gamma distributions \cite{Has14} and some HCM distributions \cite{Has16}; see e.g.\ \cite{BBLS11,AH16,BH13,Has14,Has16,MU20} for further examples. On the other hand, little is known about free L\'evy measures of these distributions; most of the former results in the literature were limited to the existence of $\Omega$. For the normal distribution, it is nonetheless shown in \cite{HST19} that the free L\'evy measure of $N(0,1)$ is of the form
\begin{equation} \label{eq:FSD}
\frac{k(x)}{|x|} \mathbf1_{\R\setminus\{0\}}(x)\,\dd x, 
\end{equation}
where $k$ is nondecreasing on $(0,\infty)$ and is non-increasing on $(-\infty,0)$, i.e., $N(0,1)$ is freely selfdecomposable. The fact that $N(0,1)$ is symmetric also implies that $k$ is an even function. The aim of this paper is to clarify further properties of the function $k$. 

\subsection{Main results and the outline of proofs}\label{sec:main}

The main result of this paper is the following two theorems on the free L\'evy measure of $N(0,1)$. 

\begin{theorem}[Analyticity and monotonicity]\label{thm:freeLevy}
The free L\'{e}vy measure of $N(0,1)$ is of the form 
\begin{equation} \label{eq:FLM_normal}
\frac{1}{\pi x^{2}} h(|x|) \mathbf{1}_{\R\setminus\{0\}}(x)\,\dd x,  
\end{equation}
 where $h\colon (0,\infty) \to (0,\infty)$ is a real analytic function with $h' <0$ on $(0,\infty)$. 
\end{theorem}
The function $h$ will be defined in \eqref{eq:gh}; it describes the height of the boundary of the domain $\Omega$ introduced in Definition \ref{def:UI} for the normal distribution $N(0,1)$.  Therefore, studying the free L\'evy measure is closely related to studying the boundary of  $\Omega$. 

Theorem \ref{thm:freeLevy} readily reproduces the known fact \eqref{eq:FSD}, because $h'<0$ implies $k'<0$, where $k(x) := x^{-1} h(x)/\pi$.  
An interesting analogous fact is that the Boolean L\'evy measure of $N(0,1)$ is also of the form $x^{-2} \tilde{h}(|x|)\,\dd x$, where $\tilde{h}$ is real analytic with negative derivative on $(0,\infty)$, see \cite[Proposition 4.2]{HNSU}. 

We then clarify the asymptotic behavior of the function $h$ of the normal distribution at zero and at infinity. 

\begin{theorem}[Asymptotic behaviors of $h$] \label{thm:fine_asymptotics_main} The function $h$ in Theorem \ref{thm:freeLevy} fullfills: 

\begin{enumerate}[label=\rm(h${}_\infty$),leftmargin=1.2cm]

\item\label{item:h(x)_infty_main} $\displaystyle h (x) = \frac{1}{e} \sqrt{\frac{\pi}{2}}x^2 e^{-\frac{x^2}{2} } (1 + O(x^{-2}))$\quad  as \quad $x\to\infty$; 
\end{enumerate}
\begin{enumerate}[label=\rm(h${}_0$),leftmargin=1.2cm]
\item \label{item:h(x)_zero_main} $\displaystyle 
h(x) = \sqrt{\log \frac1{\sqrt{2\pi}\, x} + \sqrt{ \left(\log \frac1{\sqrt{2\pi}\, x}\right)^2+ \frac{\pi^2}{4} }} + O(x^\eta)  \quad \text{as} \quad x\to0^+ \quad \text{for any } 0<\eta <1.
$
\end{enumerate}
\end{theorem}

\begin{remark}
Moreover, \ref{item:h(x)_infty_main} can be enhanced to the asymptotic expansion
\[
\displaystyle h (x) = \frac{1}{e} \sqrt{\frac{\pi}{2}}x^2 e^{-\frac{x^2}{2} } \left(1 - \frac{5}{2x^2}  -\frac{43}{8x^4} - \frac{579}{16 x^6}  - \cdots\right), \qquad x\to\infty,    
\]
 see Remark \ref{rem:asymptotic} for further details. 
\end{remark}

Proofs of the above two theorems are sketched here. 
First we construct an entire analytic continuation $\widetilde G$ of the Cauchy transform of $N(0,1)$  (Subsection \ref{subsec:entire}).  
Then we introduce a crucial supplementary domain $\Xi \subseteq \C$ on which the reciprocal Cauchy transform $\widetilde F:= 1/ \widetilde G$ is analytic  (Lemma \ref{cor:F}; note that $\widetilde F$ is a meromorphic function on $\C$). Moreover,  a simply connected domain $\Omega$ for the normal distribution as in Definition \ref{def:UI} exists as a subset of $\Xi$  (Theorem \ref{thm:omega}).  
In the construction of $\Omega $, we first identify its boundary set $\partial\Omega$ with the preimage of $\R \setminus \{0\}$ by the map $\Restr{\widetilde F}{\Xi}$ (Proposition \ref{prop:curve}). (If we go outside of $\Xi$, then the preimage of $\R \setminus \{0\}$ seems to have irrelevant connected components, see Figure \ref{fig1} below.)   This argument also implies the analyticity of $\partial\Omega$.  

A key fact is that $\partial \Omega$ is a graph of a function.  The height (from the real line) of the boundary curve  $\partial\Omega$ can be described by the function $h(x) := -\im[(\Restr{\widetilde F}{\text{cl}(\Omega)})^{-1}(x)], x>0$, which is exactly the function appearing in Theorem \ref{thm:freeLevy}. 
According to \cite[(8.1.6)]{Kerov} or \cite[(3.5)]{BBLS11}, the following ODE is satisfied by $F_{N(0,1)}$:
\begin{align}\label{eq:ODE}
F_{N(0,1)}'(z)= F_{N(0,1)}(z)(z- F_{N(0,1)}(z)), \qquad z\in \C^+. 
\end{align} 
By analytic continuation, this formula holds for $\widetilde F$ on $\Xi$ too. Going to the inverse map, a system of ODEs for $h(x)$ (and for $g(x):=\re[(\Restr{\widetilde F}{\text{cl}(\Omega)})^{-1}(x)]$) can be deduced.
The monotonicity of $h$ in Theorem \ref{thm:freeLevy} is an easy consequence of these ODEs (Proposition \ref{prop:gh}).  Formula \eqref{eq:FLM_normal} easily follows from the Stieltjes inversion (Section \ref{sec:freeLevy}). 

Considering the above, investigating the curve $\partial \Omega$ in further details will reveal fine properties of the free L\'evy measure, which results in Theorem \ref{thm:fine_asymptotics_main}. The two asymptotics \ref{item:h(x)_infty_main} and \ref{item:h(x)_zero_main} will be separately proved in expanded forms (providing finer descriptions  of $\partial \Omega$) as Theorems \ref{thm:fine_asymptotics} and \ref{thm:main_infty}, respectively. 
The proof  does not use the ODE by contrast to Theorem \ref{thm:freeLevy}.

The method for proving \ref{item:h(x)_infty_main} is based on asymptotic analyses of the reciprocal Cauchy transform and of its inverse function at infinity on a region $|\arg z| <\epsilon$. As a basis, the Laurent series asymptotic expansions of $\widetilde F$ and its inverse are  obtained in Lemmas \ref{lem:asymptotic_Cauchy} and \ref{lem:F^{-1}(z)_infty}, respectively.  In addition, we need to estimate an exponential decay of $\im[\widetilde F]$, which is invisible in the Laurent series expansion (Proof of Theorem \ref{thm:fine_asymptotics}). This decay is inherited from the tail behavior of the probability density function.  The whole method seems to be applicable to a wider class of freely infinitely distributions with unbounded support. 

The proof of \ref{item:h(x)_zero_main} is based on formula \eqref{eq:CauchyND} for $\widetilde G$. As the curve $\partial \Omega$ goes to infinity as it approaches the negative imaginary axis (cf.~Figure~\ref{fig1}), the contribution of $G(z)$ in formula  \eqref{eq:CauchyND} is negligible because of its order $O(1/z)$ on $\partial \Omega$ (and because the remaining term $-  i \sqrt{2\pi} e^{-\frac{z^2}{2}}$ goes to infinity).  The point $(\Restr{\widetilde F}{\text{cl}(\Omega)})^{-1}(x)$ is given as a unique solution $z\in \Xi$ to the equation $\widetilde F(z)=x$ that reads, for small $x>0$, $-  i \sqrt{2\pi} e^{-\frac{z^2}{2}} +O(1/z) =1/x$. A detailed analysis of this equation yields the asymptotic behavior of $h(x)= - \im(z) =-\im[(\Restr{\widetilde F}{\text{cl}(\Omega)})^{-1}(x)]$ claimed in $(h_0)$. A key observation in this analysis is that $\partial \Omega$ can be well approximated by the curve $\partial\Xi$ near the imaginary axis, cf.~Figure~\ref{fig1}.

\section{The Cauchy transform of the normal distribution}\label{sec:cauchy}

In this section, we analyze the analytic continuation of the Cauchy transform and its reciprocal, and then describe the boundary of $\Omega$ as a graph of an analytic function, where  $\Omega$ is the domain appearing in Definition \ref{def:UI} for $N(0,1)$. 

\subsection{Entire analytic continuation of the Cauchy transform} \label{subsec:entire}

We simplify the notation of the Cauchy transform of the normal distribution into
\begin{equation}\label{eq:cauchy}
G (z):= G_{N(0,1)}(z) =  \int_{-\infty}^\infty \frac{1}{z-x} \cdot \frac{1}{\sqrt{2\pi}} e^{-\frac{x^2}{2}}\dd x, \qquad z \in \C\setminus\R. 
\end{equation}
A well known fact is that the Cauchy transform $\Restr{G}{\C^+}$ has an analytic continuation to $\C$ (denoted by $\widetilde{G}$) and, on the lower half-plane, the formula 
\begin{align}\label{eq:CauchyND}
\widetilde{G}(z)=G(z)- 2\pi i \cdot \frac{1}{\sqrt{2\pi}} e^{-\frac{z^2}{2}}, \qquad z\in \C^-
\end{align}
holds, see e.g.\ \cite[Theorem 1.2]{Gre60}. On the other hands, due to \cite[p.~362]{Gre60} and the identity theorem, we have
\begin{align}\label{eq:CauchyND2}
\widetilde{G}(z)=e^{-\frac{1}{2}z^2}\left[-i\sqrt{\frac{\pi}{2}}+\sqrt{2} \int_0^{z/\sqrt{2}} e^{t^2}\dd t \right], \qquad z\in \C.
\end{align}
In particular, we have
\begin{equation} \label{eq:Stieltjes}
\im[\widetilde{G}(x)] = - \sqrt{\frac{\pi}{2}}e^{-\frac{x^2}{2}} <0,  \qquad x\in \R,  
\end{equation}
which can also be deduced from the Stieltjes inversion formula. 

Obviously, the reciprocal Cauchy transform $F_{N(0,1)} = 1/G_{N(0,1)}$ analytically extends to the meromorphic function on $\C$
\[
\widetilde{F}(z) := \frac{1}{\widetilde{G}(z)}. 
\]
In view of \eqref{eq:F} and \eqref{eq:Stieltjes}, poles of $\widetilde{F}$ do not exist in $\C^+ \cup \R$. It seems that $\widetilde{F}$ has poles on $\C^-$, see Figure \ref{fig3} below; however, we mostly work on $\widetilde{F}$ in subdomains where $\widetilde{F}$ turns out to have no poles, so that analysis of poles will be rather out of scope of this paper. 

\subsection{Behavior of the Cauchy transform on extended domains}

As preparatory steps, we investigate the behavior of $\widetilde G$ and $\widetilde F$ on the imaginary axis (Lemma \ref{lem:F(ix)}), asymptotic behavior as  $z\to \infty,  - (1/4)\pi + \epsilon < \arg z < (5/4)\pi - \epsilon$ (Lemma \ref{lem:asymptotic_Cauchy} and Lemma \ref{lem:F-transform}) and then their behaviors  on a domain $\Xi$  (Lemma \ref{cor:F}). 

The next fact corresponds to the exceptional case $c=0$ excluded in \cite[Lemma 3.6]{BBLS11}. 
\begin{lemma}\label{lem:F(ix)}
\begin{enumerate}[label=\rm(\arabic*),leftmargin=1cm]
\item $\displaystyle \lim_{\substack{x\in \R\\ x\rightarrow\infty}} \widetilde{F}(ix)=i\infty$\quad and \quad$\displaystyle \lim_{\substack{x\in \R\\ x\rightarrow-\infty}} \widetilde{F}(ix)=i0$.
\item $\widetilde{F}$ is a bijection from $i\R$ onto $i(0,\infty)$.
\end{enumerate}
\end{lemma}

\begin{proof}
Obviously it suffices to verify the equivalent assertions
\begin{enumerate}[label=(\roman*),leftmargin=1.2cm]
\item\label{item:G1} $\displaystyle \lim_{\substack{x\in \R\\ x\rightarrow\infty}} \widetilde{G}(ix)=i0$\quad and \quad$\displaystyle \lim_{\substack{x\in \R\\ x\rightarrow-\infty}} \widetilde{G}(ix)=-i\infty$; 
\item\label{item:G2} $\widetilde{G}$ is a bijection from $i\R$ onto $i(-\infty,0)$.
\end{enumerate}

\noindent
{\bf Proof of \ref{item:G1} and \ref{item:G2}.}
\begin{enumerate}[label=(\roman*),leftmargin=0.8cm]

\item The first limit is clear from \eqref{eq:cauchy}. By \eqref{eq:CauchyND}, we have
\begin{align}\label{eq:G(ix)}
\widetilde{G}(ix)=G(ix)- 2\pi i \cdot \frac{1}{\sqrt{2\pi}} e^{\frac{x^2}{2}} \rightarrow -i\infty,  \qquad x\rightarrow -\infty.
\end{align}

\item By \eqref{eq:CauchyND2}, for $x\in \R$, we get
\begin{align*} 
\rho(x):= i\widetilde{G}(ix)=\sqrt{\frac{\pi}{2}} e^{\frac{x^2}{2}} \left[1- \frac{2}{\sqrt{\pi}} \int_0^{x/\sqrt{2}} e^{-t^2} \dd t \right] \in (0,\infty),
\end{align*}
 It follows from (1) that $\lim_{x\rightarrow\infty} \rho(x)=0$ and $\lim_{x\rightarrow-\infty} \rho(x)=\infty$, and  hence $\widetilde{G}$ maps $i\R$ onto $i(-\infty,0)$.  It remains to establish the monotonicity of $\rho$. Since
\begin{align*}
\rho'(x)=\sqrt{2}x e^{\frac{x^2}{2}} \int_{x/\sqrt{2}}^\infty e^{-t^2}\,\dd t - 1,
\end{align*}
it obviously follows that $\rho'(x)<0$ for all $x\le 0$. Some calculus also shows that $\rho'(x)<0$ for all $x>0$. Consequently, $\rho$ is strictly decreasing on $\R$, and therefore $\widetilde{G}$ is a bijection from $i\R$ onto $i(-\infty,0)$. 
\end{enumerate}
\vspace{-7mm}
\end{proof}

The Cauchy transform is well known to have an asymptotic expansion as $z\to \infty,$ $ \epsilon < \arg z < \pi -\epsilon$ for any fixed $\epsilon \in (0,\pi)$, see e.g.\ \cite[Theorem 3.2.1]{Akh65}. For the normal distribution, we can see that the asymptotic expansion holds in the larger domain 
\begin{align*}
D_{\epsilon}:=\left\{z\in \C\setminus\{0\}: \arg(z) \in \left(-\frac{\pi}{4}+\epsilon, \ \frac{5}{4}\pi - \epsilon \right) \right\},    
\end{align*}
with $\epsilon\in(0,\pi/4)$ arbitrary but fixed. To state the formula, we denote by $\{m_k\}_{k\ge0}$ the moment sequence of $N(0,1)$, i.e.\ 
\[
m_0=1; \qquad m_n= \begin{cases} (n-1)!! & \text{if~$n$ is even}, \\ 0& \text{if $n$ is odd}
\end{cases}
 \quad(n \in \N).
\]

\begin{figure}[b!]
\begin{center}
\begin{overpic}[width=5cm]{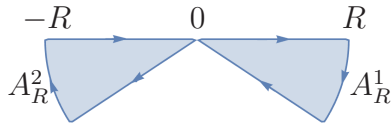}
\put(4,28){$-R$}
\put(48,28){$0$}
\put(88,28){$R$}
\put(0,10){$A_R^2$}
\put(90,10){$A_R^1$}
\end{overpic}
\caption{Contour integrals} \label{fig:D}
\end{center}
\end{figure}

\begin{lemma}\label{lem:asymptotic_Cauchy}
For any fixed $N \in \N$ and fixed $\epsilon \in (0, \frac{\pi}{4})$, the asymptotic expansions
\begin{align}
\widetilde{G}(z) &=  \sum_{n=0}^{N-1} \frac{m_{2n}}{z^{2n+1}} + O(z^{-2N-1})   \quad \text{and}   \label{eq:asymptotic_Cauchy1} \\
\widetilde{G}' (z) &=  - \sum_{n=0}^{N-1} \frac{(2n+1)m_{2n}}{z^{2n+2}} + O(z^{-2N-2})  \label{eq:asymptotic_Cauchy2}
\end{align}
hold as $z\rightarrow\infty$ with $z\in D_{\epsilon}$. 
\end{lemma}
\begin{proof}
For $0<\epsilon <\pi/4$ we first observe that
\begin{align}\label{eq:another_rep}
G (z)=\int_{\partial D_\epsilon} \frac{1}{z-w}\frac{1}{\sqrt{2\pi}} e^{-\frac{w^2}{2}}\,\dd w, \qquad z\in \C^+. 
\end{align}
This is an easy consequence of Cauchy's integral formula applied to the region(s) in Figure \ref{fig:D}
and the fact that the contour integrals over the arcs $A_R^1:=\{Re^{i\theta}: -\frac{\pi}{4}+\epsilon \le \theta \le 0\}$ and $A_R^2:=\{Re^{i\theta}: \pi \le \theta \le \frac{5\pi}{4}-\epsilon\}$ converge to zero as $R \to \infty$. Indeed, for $A_R^1$
\begin{align*}
\left|   \int_{A_R^1} \frac{1}{z-w}\frac{1}{\sqrt{2\pi}} e^{-\frac{w^2}{2}}\,\dd w   \right| 
& =\left|  \int^{-\frac{\pi}{4}+\epsilon}_0 \frac{1}{z-Re^{i\theta}}\frac{1}{\sqrt{2\pi}} e^{-\frac{1}{2}R^2e^{2i\theta}} iRe^{i\theta}  \,\dd \theta \right| \\
&\le  \frac{1}{\sqrt{2\pi}}\int_{-\frac{\pi}{4}+\epsilon}^0 \frac{R}{R-|z|} e^{-\frac{1}{2}R^2\cos2\theta}  \,\dd \theta \\
&\le  \frac{(\pi/4 -\epsilon)}{\sqrt{2\pi}}  \frac{R}{R-|z|} e^{-\frac{1}{2}R^2\sin 2\epsilon}  \to 0\qquad (R\to\infty).  
\end{align*}
The integral over $A_R^2$ is similarly estimated.  

The remaining arguments are analogous to the standard one for $ \epsilon < \arg z < \pi -\epsilon$, see e.g.\ the proof of \cite[Theorem 3.2.1]{Akh65}. For the reader's convenience the rest of the proof is included in Appendix \ref{appendix}. 
\end{proof}

By Lemma \ref{lem:asymptotic_Cauchy}, the analytic extension $\widetilde{F}$ has no poles on $D_{\epsilon,R}:= D_\epsilon \cap\{z: |z|>R\}$ for sufficiently large $R>0$ and 
\begin{align}\label{eq:F_asympt}
\widetilde{F}(z)=z(1+o(1)), \qquad z\rightarrow\infty, \ z\in D_{\epsilon}.
\end{align}

The next two lemmas are basic ingredients to construct and analyze a compositional inverse function of $\widetilde{F}$. 
\begin{lemma}\label{lem:F-transform}
For any $0< \epsilon< \epsilon' < \pi/4$   there exist $R>0$ such that $\widetilde{F}$ is univalent on $D_{\epsilon,R}$ and $D_{\epsilon',R'}\subset\widetilde{F}(D_{\epsilon,R})$, where $R' := [1+\sin(\epsilon' - \epsilon)]R$. 
\end{lemma}

\begin{proof}
\textbf{Part 1: $D_{\epsilon',R'}\subset\widetilde{F}(D_{\epsilon,R})$.} Due to \eqref{eq:F_asympt}, there exists an $R>0$ such that 
\[
|\widetilde{F}(z) -z | < \sin (\epsilon'-\epsilon) |z|, \qquad z\in \partial D_{\epsilon, R}. 
\]

Let $d(z,D_{\epsilon',R'})$ stand for the distance between $z\in \partial D_{\epsilon,R}$ and the domain $D_{\epsilon',R'}$. It is elementary to verify that 
\begin{align}\label{eq:distance}
d(z, D_{\epsilon',R'})\ge \sin(\epsilon'-\epsilon) |z|, \qquad z\in \partial D_{\epsilon,R}, 
\end{align}
which implies that the curve $\widetilde{F}(\partial D_{\epsilon,R})$ does not intersect with $D_{\epsilon',R'}$ and every point of $D_{\epsilon',R'}$ has rotation number 1 with respect to this curve (viewed as a closed curve in the Riemann sphere), and hence $D_{\epsilon',R'}\subset\widetilde{F}(D_{\epsilon,R})$.

\vspace{2mm}
\noindent
\textbf{Part 2: Univalence of $\widetilde{F}$.}   In order to resort to the Noshiro--Warschawski criterion (see e.g. \cite[Proposition 1.10]{Pom92} or the original articles \cite{No34,Wa35}), we estimate the derivative $\widetilde{F}'$ on $D_{\epsilon,R}$. 
Take $0< \eta <\epsilon< \pi/4$. By Lemma \ref{lem:asymptotic_Cauchy}, we have
\[
\widetilde{F}'(z)=-\frac{\widetilde{G}'(z)}{\widetilde{G}(z)^2} \sim 1 \quad \text{as} \quad z\rightarrow\infty \quad\text{with}\quad z\in D_\epsilon, 
\]
and therefore, we can take $R_0>0$ large enough so that $\re[\widetilde{F}']\ge 1/2$ on $D_{\epsilon,R_0}$. 

Because $D_{\epsilon,R_0}$ is not convex, we introduce supplementary convex domains. 
Let $\ell$ be the half-line starting from the point $4R_0 e^{i(-\frac{\pi}{4} +\epsilon)}$, passing $4R_0 i$ and going to infinity. 
Let $U$ be the domain that has the boundary $\ell \cup \{ re^{i(-\frac{\pi}{4} +\epsilon)}: r\ge 4R_0\}$ and contains the point $4R_0(1+i)$.  Let $V$ be the reflection of $U$ with respect to the imaginary axis. Since $U$ and $V$ are convex domains contained in $D_{\epsilon,R_0}$, the Noshiro--Warschawski criterion implies that $\widetilde{F}$ is univalent in $U$ and $V$. Choosing $R=8R_0$ and using the fact that $\widetilde{F}$ is close to the identity map, i.e.\ $\widetilde{F}(z) = z(1 +o(1))$, we can conclude that $\widetilde{F}$ is univalent in $D_{\epsilon,R}$ for large $R_0>0$. 
\end{proof}

\begin{lemma}\label{cor:F}
Let $\Xi \subseteq\C$ be the domain with boundary 
\[
\partial \Xi=C_{\pi} \cup C_{-\pi},
\]
where $C_{\pm \pi}:=\{re^{i\theta} : r>0, \ -\pi <\theta<0, \ r^2\sin 2\theta=\pm\pi\}$.  Then the function $\widetilde{G}$ has no zeros in $\Xi \cup \partial \Xi$, and therefore $\widetilde{F}$ is analytic in $\Xi \cup \partial \Xi$. Moreover, the function $\Restr{\widetilde{F}}{\Xi \cup \partial \Xi}$ satisfies 
\begin{enumerate}[label=\rm(\arabic*),leftmargin=1.2cm]
\item\label{item:F1} $\im [\widetilde{F}(z)]<0$ for all $z\in \partial \Xi$,
\item $\re [\widetilde{F}(z)]>0$ for all $z\in \Xi \cap \{z: \re(z)>0\}$, 
\item\label{item:F3} $\displaystyle\lim_{\substack{\im(z) \to- \infty \\ z\in \Xi}}\widetilde{F}(z) =0$.
\end{enumerate}
\end{lemma}

\begin{proof} 
It suffices to establish
\begin{enumerate}[label=(\roman*),leftmargin=1.2cm]
\item\label{item:GG1} $\im[\widetilde{G}(z)]>0$ for all $z\in \partial \Xi$, 
\item\label{item:GG2} $\re [\widetilde{G}(z)]>0$ for all $z\in \Xi\cap \{z: \re(z)>0\}$, 
\item\label{item:GG3} $\displaystyle \lim_{\substack{\im(z) \to- \infty \\ z\in \Xi}}\widetilde{G}(z) =\infty$.   
\end{enumerate}
Note that \ref{item:GG1} and \ref{item:GG2} together with Lemma \ref{lem:F(ix)} (and the fact that $\im[\widetilde{G}]<0$ on $\R$) imply that $\widetilde G$ has no zeros in $\Xi \cup \partial \Xi$.

\vspace{3mm}
\noindent
{\bf Proof of \ref{item:GG1}--\ref{item:GG3}.}  The proofs are based on separate analyses of the two terms of formula \eqref{eq:CauchyND}. 
\begin{enumerate}[label=(\roman*),leftmargin=0.8cm]
\item If $z=re^{i\theta}\in \partial \Xi$, that is, $r^2\sin 2\theta=\pm \pi$ and $-\pi <\theta<0$, then 
\[
e^{-\frac{1}{2}z^2}= \mp ie^{-\frac{r^2}{2}\cos 2\theta},
\]
respectively. Hence for such $z$, we get
\[
\widetilde{G}(z)=G(z)- 2\pi i \cdot \frac{1}{\sqrt{2\pi}} \cdot \left(\mp ie^{-\frac{r^2}{2}\cos 2\theta}\right) \in \C^+.
\]

\item For $z \in \C^+\cap \{z: \re(z)>0\}$, according to \cite[Lemma 3.1]{BH13}, we have $\re[\widetilde{G}(z)] >0$ .  For $z=re^{i\theta} \in \Xi \cap \C^- \cap \{z: \re(z)>0\}$ we have $-\pi<r^2\sin 2\theta<0$ and hence the number
\[
e^{-\frac{z^2}{2}}= e^{-\frac{r^2}{2}\cos 2\theta} e^{-i \cdot \frac{r^2}{2}\sin 2\theta}
\]
has positive imaginary part.  Again by  \cite[Lemma 3.1]{BH13} $G(z)$ also has positive real part, so that  we conclude $\re[\widetilde{G}(z)]>0$. For $z= x \in (0,\infty)$,  the continuity of $G$ implies $\re[G(x -i0)]\ge0$ $(x>0)$, so the conclusion still holds. 

\item As $\im(z) \to- \infty, z\in \Xi$, the argument of $z$ tends to $-\pi/2$, so that $e^{-\frac{z^2}{2}}$ tends to $\infty$. On the other hand, $G(z)$ tends to zero. We are done.  
\end{enumerate}
\vspace{-7mm}
\end{proof}

\begin{remark}
The domain $\Xi \cap \C^-$ is the connected component of 
\[
\left\{ z \in \C^-: \im\left[ -i e^{-\frac{z^2}{2}} \right]<0 \right\}
\] 
that contains the negative imaginary axis, see Figure \ref{fig:xi}. 

\begin{figure}[t!]
\begin{center}
\begin{overpic}[width=5cm]{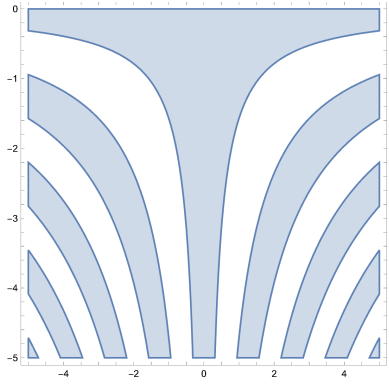}
\end{overpic}
\caption{The region $\left\{ z \in \C^-: \im\left[ -i e^{-\frac{z^2}{2}} \right]<0 \right\}$} \label{fig:xi}
\end{center}
\end{figure}


\end{remark}

\subsection{Construction of $\Omega$ and description of its boundary} \label{subsec:boundary}
Using results in the previous subsection, we construct $\Omega$ and describe its boundary (Theorem \ref{thm:omega}). The boundary turns out to be the graph of an analytic function. We first provide supplementary facts. 

\begin{proposition} \label{prop:curve}
For every $x>0$ there exists a unique point $H(x)$ in $\Xi \cap \C^- \cap \{z: \re(z)>0\}$ such that $\widetilde F(H(x))=x$. The curve $p_0^+ =\{H(x): x>0\}$ is real analytic and is mapped by $\widetilde{F}$ bijectively onto $(0,\infty)$. 
By symmetry, the curve $p_0^-$ which is the reflection of $p_0^+$ with respect to the imaginary axis is mapped by $\widetilde{F}$ bijectively onto $(-\infty, 0)$. 
\end{proposition}
\begin{proof}
Let $c_R$ be the simple closed curve in the Riemann sphere consisting of $i[-\infty, 0], \{z \in \partial\Xi: 0< \re(z) < R \}, \{R +i y: -\frac{\pi}{2R} \le y \le 0\}$ and $[0, R]$. 
We take $R$ sufficiently large so that $|\widetilde{F}(z) -z| < \frac1{2}|z|$ for $|z| \ge R, z \in c_R$. By Lemma \ref{lem:F(ix)},  Lemma \ref{cor:F} \ref{item:F1}, \ref{item:F3} and the fact that $\im[F]>0$ on $\R$, we can observe that every point of $(0, \frac1{2}R)$ is surrounded by the curve $\widetilde{F}(c_R)$ exactly once. This implies that for every $x \in (0, \frac1{2}R)$ there exists a unique point $H(x)$ in the Jordan domain surrounded by $c_R$ such that $\widetilde{F}(H(x)) =x$. Because $R$ is arbitrary as long as sufficiently large, for every $x \in (0, \infty)$ there exists a unique point $H(x)$ in the domain surrounded by the curve $i[-\infty, 0]\cup \{z \in \partial\Xi: 0< \re(z) < \infty \} \cup [0,\infty)$ such that $\widetilde{F}(H(x)) =x$. It remains to prove the analyticity of the function $H$. First note that $\widetilde F'(z) \ne 0$ holds on $p_0^+$; otherwise the point $H(x)$ would not be unique. Therefore, $\widetilde F$ is locally bijective at each point $H(x)$ and hence its inverse function $H$ is also analytic in a complex neighborhood of each point $x>0$. 
\end{proof}

We then study properties of analytic functions
\begin{equation} \label{eq:gh}
g(x) := \re[H(x)] \qquad \text{and} \qquad h(x):= -\im[H(x)].
\end{equation}
Note that this function $h$ will turn out to coincide with $h$ in Theorem \ref{thm:freeLevy}, see \eqref{eq:inverse_F} and \eqref{eq:Stieltjes_h}. 

Obviously, we have $g, h >0$. 
By analytic continuation,  Equation \eqref{eq:ODE} easily extends to 
\begin{align}\label{eq:ODE2}
\widetilde{F}'(z)= \widetilde{F}(z)(z- \widetilde{F}(z)), \qquad z\in \Xi. 
\end{align}
Because $H$ is the compositional inverse map of $\Restr{\widetilde{F}}{p_0^+}$ and $\widetilde{F}' (z)$ does not vanish on $p_0^+$,  
 the ODE \eqref{eq:ODE2} restricted to $p_0^+$ entails the ODE $H'(x) = \frac{1}{x(H(x)-x)}$, which is equivalent to  
\begin{align}
g'(x)&=\frac{g(x)-x}{x((g(x)-x)^2+h(x)^2)} \quad \text{and} \label{eq:g} \\
h'(x)&=-\frac{h(x)}{x((g(x)-x)^2+h(x)^2)}.  \label{eq:h}
\end{align}
Using this ODE we provide some properties of $g$ and $h$ below. Some of the results will be made much finer in Section \ref{sec:freeLevy}. 

\begin{proposition}\label{prop:gh}
The following hold: 
\begin{align*}
g'>0 \quad \text{and} &\quad h'<0 \quad \text{on} \quad (0,\infty);  \\ 
 \lim_{x\to \infty} g(x) =\infty, \quad \lim_{x\to \infty} h(x) =0, &\quad  \lim_{x\to 0^+} g(x) =0 \quad \text{and}\quad  \lim_{x\to 0^+} h(x) =\infty. 
\end{align*}
 \end{proposition}
\begin{proof}
Equation \eqref{eq:h} readily implies $h'(x)<0$. Because of \eqref{eq:F_asympt} and the construction of $p_0^+$, we can conclude that $\lim_{x\to\infty} g(x)=\infty$, which associates the limit $\lim_{x\to\infty} h(x)=0$, 
because $g(x) - i h(x) \in \Xi \cap \C^-$. To show $g'(x)>0$ it is convenient to introduce the function $ \omega(x) := g(x)-x$. Suppose to the contrary that $g'$ takes a nonpositive value. 
Because $g(x) \to \infty$ as $x \to\infty$, there is at least a strictly increasing sequence converging to $\infty$ on which $g'$ takes positive values. 
Therefore, we can find $x_0, x_1 \in (0,\infty)$ with $x_0 <x_1$ such that $g'(x_0)=0$ and $g'(x)>0$ for all $x \in (x_0,x_1)$. In terms of $\omega$ the former reads $\omega'(x_0)=-1$. 
In view of \eqref{eq:g} it also follows that $\omega(x_0)=0$ and $\omega(x)>0$ for $x \in (x_0,x_1)$.  These (in)equalities obviously contradict. The proof of $g'(x)>0$ is thus complete. 

It remains to prove the last two limits.  It suffices to establish $\lim_{x\to 0^+} h(x) =\infty$, because then $\lim_{x\to 0^+} g(x) =0$ follows from the fact that $g(x) - ih(x) \in \Xi$.  

Suppose to the contrary that $\beta := \lim_{x\to 0^+} h(x) <\infty$. We set $\alpha :=  \lim_{x\to 0^+} g(x) \in [0,\infty)$. Then $\alpha-i\beta \in \Xi \cup \partial \Xi$. Taking the limit in the formula $\widetilde F(g(x) - i h(x)) = x$, we get $\widetilde F(\alpha -i \beta) =0$, a contradiction to the fact that $\widetilde F$ does not have zeros on $\Xi \cup \partial \Xi$.
\end{proof}

\begin{theorem}\label{thm:omega}
We set  
\begin{align*}
\Omega:=\{x+iy \in \C: x,y\in \R,\ x\neq0,\ y>f(x)\} \cup i\R,
\end{align*}
where $f\colon \R\setminus\{0\}\rightarrow (-\infty,0)$ is the real analytic function determined by $f(x) = -h \circ g^{-1}(|x|)$. Then $\Omega$ fulfills the requirement of Definition \ref{def:UI} for $N(0,1)$, i.e., $\widetilde{F}$ is an analytic bijection from $\Omega$ onto $\C^+$. Moreover, $\widetilde{F}$ is a homeomorphism from $\text{cl}(\Omega)$ onto $(\C^+\cup\R) \setminus \{0\}$. 
\end{theorem}

\begin{proof}

First note that $\Omega$ is exactly the domain that has boundary $p_0^+ \cup p_0^- $ and contains $\C^+$ as a subset. 

Because $\partial \Xi \cup \{\infty\}$ is not a simple closed curve in the Riemann sphere, we introduce an approximating simple closed curve $\gamma_R$ in the Riemann sphere consisting of $\{z \in p_0^+ \cup\{\infty\} \cup p_0^-: |\re(z)| \le R \}$,  the semicircle $\{R e^{i\theta}: 0 \le \theta \le \pi\}$ and two vertical line segments $R+i[-f(R), 0]$ and $-R+i[-f(R), 0]$, where $\infty$ stands for the point corresponding to $\lim_{x\to0^+} H(x)$ and $R>0$. Note that, as a consequence of Lemma \ref{cor:F} \ref{item:F3}, $\widetilde{F}$ can be regarded as a continuous function on closure (in the Riemann sphere) of the Jordan domain surrounded by $\gamma_R$. This fact is important in the next paragraph when we apply the Darboux theorem. 

We prove that $\widetilde{F}$ is univalent on  $\gamma_R$ for sufficiently large $R>0$. Because of the symmetry with respect to the imaginary axis, it suffices to prove the univalence on $\gamma_R \cup \{z: \re(z)>0\}$. First, $\widetilde{F}$ is univalent on $p_0^+$ by the construction of $p_0^+$. Also, we can take $r,R>0$ with $ r, g(r) < R/4$ large enough so that, by Lemma \ref{lem:F-transform},  $\widetilde{F}$ is univalent in $\gamma_R \cap \{x+iy: |x| \ge g(r) \text{~or~} y \ge r\}$ and that, by \eqref{eq:F_asympt},  $|\widetilde{F}(z)- z| < \frac{|z|}{2}$ holds for all $|z| \ge R, z \in \gamma_R$. In this situation one can see the univalence on the whole $\gamma_R$. This furthermore implies that $\widetilde{F}$ is a bijection from the Jordan domain surrounded by $\gamma_R$ onto the Jordan domain surrounded by $\widetilde{F}(\gamma_R)$ according to the Darboux theorem,  see e.g.\ \cite[Exercise 2.3-4]{Pom92} or \cite[Corollary 9.5]{Pom75}.\footnote{These references assume the Jordan curve to be contained in $\C$. Although $\gamma_R$ passes $\infty$, it can be suitably mapped to the complex plane, e.g.\ by the map $T\colon (\C \cup\{\infty\})\setminus\{c\} \to \C, T(z)= 1/(z-c)$ with a fixed $c \in \C\setminus \Xi$. Then \cite[Exercise 2.3-4]{Pom92} or \cite[Corollary 9.5]{Pom75} can be applied to the function $\widetilde{F} \circ T^{-1}$ which maps the Jordan curve $T(\gamma_R) \subset \C$ bijectively to the Jordan curve $\widetilde{F}(\gamma_R) \subset \C$.} 
By letting $R\to\infty$, we conclude that $\widetilde{F}$ is an analytic bijection from $\Omega$ onto $\C^+$. 

Finally, by Carath\'eodory's theorem \cite[Theorem 2.6]{Pom92} and the fact that $\gamma_R$ is a Jordan curve in the Riemann sphere, $\widetilde{F}$ is a homeomorphism from $\text{cl}(\Omega)$ onto $(\C^+\cup\R) \setminus \{0\}$. 
\end{proof}

\begin{remark} \label{rem:UI}   Theorem \ref{thm:omega} implies the fact $N(0,1) \in {\bf UI}$ that follows easily from the work \cite{BBLS11}, where probability measures $\mu_c, c\in (-1,0)$, called the Askey-Wimp-Kerov distributions, are shown to be in \textbf{UI}. Because class \textbf{UI} is weakly closed (see \cite[p. 2763]{AH13}) and $N(0,1)$ is the weak limit of $\mu_c$ as $c\to0^-$, we conclude that $N(0,1)$ also belongs to \textbf{UI}. The strategy there for proving $\mu_c \in \mathbf{UI}$ was to construct, for each $t>0$, a simple curve $p_t^c$ which is symmetric with respect to the imaginary axis, passing through a unique point in $i\R$ and so that the reciprocal Cauchy transform maps $p_t^c$ bijectively onto $\R+it$, see \cite[Lemma 3.8]{BBLS11} for the construction of $p_t^c$. 

However, the domain $\Omega$ was not investigated in \cite{BBLS11}. By contrast, our method directly constructed the boundary $p_0^+ \cup p_0^-$ of $\Omega$ as the preimage of $\R \setminus\{0\}$ by the map $\Restr{\widetilde F}{\Xi}$ (for the interested reader, the curves $\im[\widetilde{F}]=t$ for different $t$'s are shown in Figure \ref{fig2}). Further details on the boundary $p_0^+ \cup p_0^-$ will be clarified in Theorems \ref{thm:fine_asymptotics} and \ref{thm:main_infty} below. 
\end{remark}

\begin{figure}[h!]
\begin{center}
\begin{minipage}{0.45\hsize}
\begin{center}
\begin{overpic}[width=6cm]{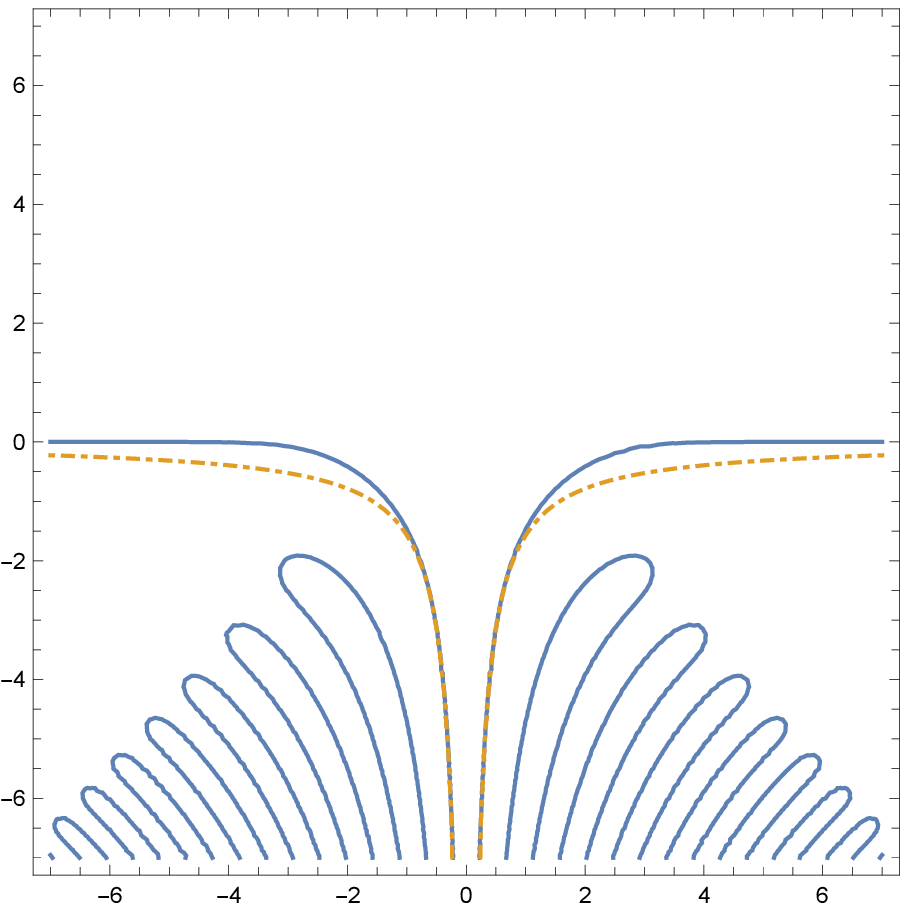}
\put(20,55){$p_0^-$}
\put(75,55){$p_0^+$}
\end{overpic}
\end{center}
\caption{the curves $\im[\widetilde{F}]=0$ (undashed) and $\partial \Xi$ (dashed). Outside $\Xi$ there are more preimages (different from $p_0^\pm$) of $\R \setminus\{0\}$ by the mapping $\widetilde{F}$. When the real part is sufficiently small, the curves $p_0^+ \cup p_0^-$ and $\partial \Xi$ are close. } \label{fig1}
\end{minipage}
\hspace{5mm}
\begin{minipage}{0.40\hsize}
\begin{center}
\begin{overpic}[width=6cm]{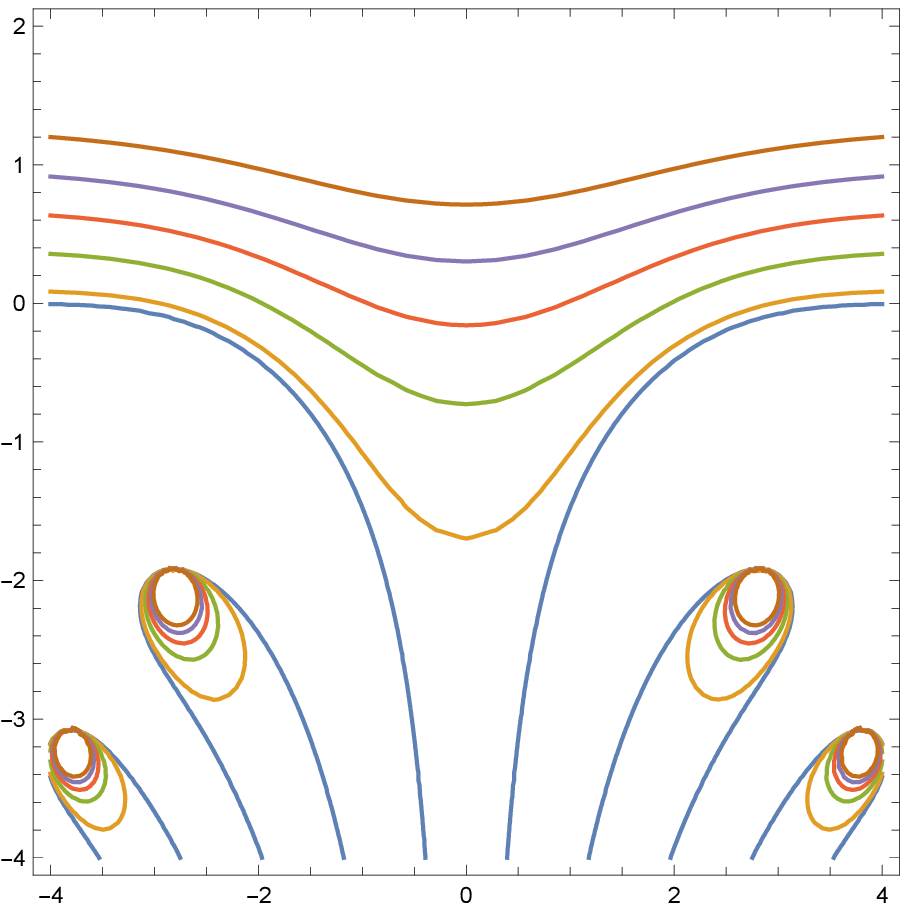}
\put(20,55){$p_0^-$}
\put(75,55){$p_0^+$}
\end{overpic}
\end{center}
\caption{the curves $\im[\widetilde{F}]=t$ for $t=0$ (blue), $t=0.1$ (yellow), $t=0.4$ (green), $t=0.7$ (red), $t=1$ (purple), $t=1.3$ (brown).  \vspace{9mm} }\label{fig2} 
\end{minipage}
\end{center}

\begin{center}
\begin{minipage}{0.45\hsize}
\begin{center}
\begin{overpic}[width=6cm]{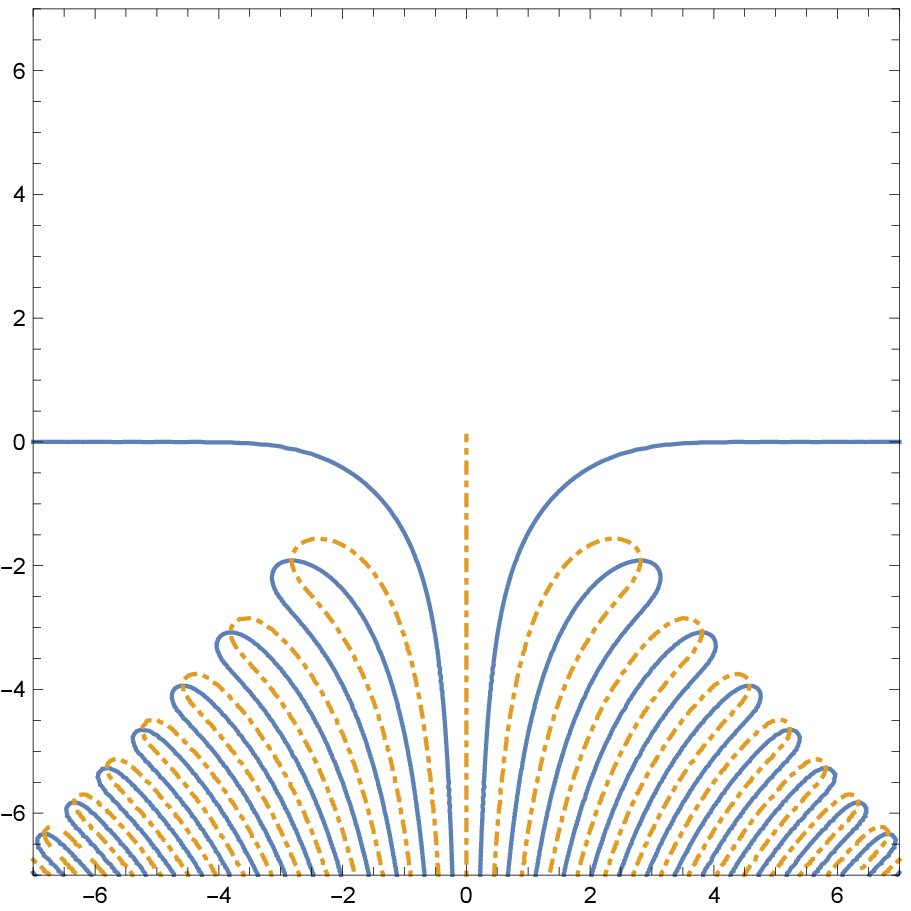}
\put(20,55){$p_0^-$}
\put(75,55){$p_0^+$}
\end{overpic}
\end{center}
\caption{the curves $\im[\widetilde{G}]=0$ (undashed) and $\re[\widetilde{G}]=0$ (dashed). The intersection of them is the poles of $\widetilde{F}$. } \label{fig3}
\end{minipage}
\hspace{5mm}
\begin{minipage}{0.40\hsize}
\begin{center}
\vspace{8mm}
\begin{overpic}[width=6cm]{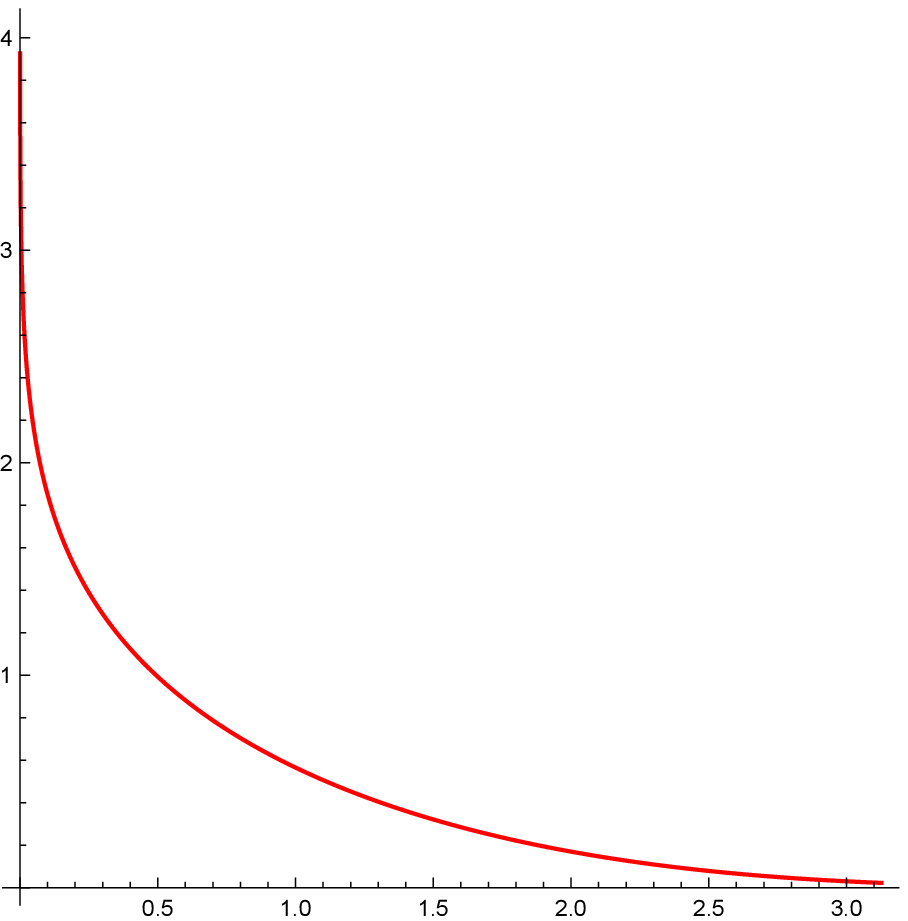}
\end{overpic}
\end{center}
\caption{the graph of $h(x)$. }\label{fig4}   \vspace{18mm} 
\end{minipage}
\end{center}
\end{figure}

\section{Proofs of the main results}\label{sec:freeLevy}

According to Theorem \ref{thm:omega}, we can define the analytic compositional inverse function $\widetilde{F}^{-1}\colon\C^+ \to \Omega$ which extends to a homeomorphism from $(\C^+\cup\R) \setminus \{0\}$ onto $\text{cl}(\Omega)$. It is even analytic on  $\R\setminus \{0\}$ with values 
\begin{equation}\label{eq:inverse_F}
\widetilde{F}^{-1} (x)  = \text{sign}(x)g(|x|) - i h(|x|)
\end{equation}
 according to the previous section and by symmetry. On the other hand, by Lemma \ref{lem:F-transform}, the function $\widetilde{F}$ is a bijection from $D_{\epsilon,R}$ onto its range that contains $D_{\epsilon',R'}$ for sufficiently large $R$. This allows us to get an analytic continuation of $\widetilde{F}^{-1}$ to the domain $(\C^+\cup\R) \cup D_{\epsilon',R'} \setminus \{0\}$. 

\begin{proof}[Proof of Theorem \ref{thm:freeLevy}]
Recall that the Voiculescu transform $\widetilde{\varphi}(w) := \widetilde{F}^{-1}(w) -w$ is an analytic function from $(\C^+\cup \R)\setminus\{0\}$ to $\C^-\cup \R$. In the Pick--Nevanlinna representation 
\begin{align}\label{eq:VT}
\widetilde{\varphi}(z)=b+\int_\R \frac{1+zx}{z-x} \, \tau(\dd x),  \qquad z \in \C^+, 
\end{align}
where $b\in \R$ and $\tau$ is a finite measure on $\R$, the Stieltjes inversion formula (see e.g.\ \cite[Theorems F.2, F.6]{Sch12}) yields 
\begin{equation} \label{eq:Stieltjes_h}
(1+x^2)\mathbf1_{\R\setminus\{0\}}(x)\tau(\dd x)=-\frac{1}{\pi} \im [\widetilde{\varphi}(x)]\,\dd x = \frac{1}{\pi} h(|x|)\,\dd x. 
\end{equation}
According to \eqref{eq:freeLM} the free L\'{e}vy measure of $N(0,1)$ is given as desired. 
\end{proof}

\begin{remark}
Because $N(0,1)$ is symmetric, $\widetilde{\varphi}(i)  \in i \R$ and hence the number $b = \re[\widetilde{\varphi}(i)]$ vanishes. Moreover, we can show that the semicircular component $\tau(\{0\})$ vanishes too. Indeed, according to \eqref{eq:G(ix)}, we have $\widetilde{F}(ix) \sim i (1/\sqrt{2\pi}) e^{-x^2/2}$ as $x\to -\infty$. Using the formula $\tau(\{0\})= \lim_{y \to 0^+} iy \widetilde{\varphi}(iy)$ (see e.g.\ \cite[Theorem F.2]{Sch12}) we deduce that 
\[
\tau(\{0\}) = \lim_{y\to0^+} iy \widetilde{F}^{-1}(iy) = \lim_{x\to -\infty} \widetilde{F}(ix) ix =0 
\]
as desired. 
\end{remark}

Let $\kappa_n$ be the $n$-th free cumulant of $N(0,1)$ below. Because $N(0,1)$ is symmetric, the odd free cumulants vanish. According to \cite[p. 3683]{BBLS11}, some even free cumulants are given as follows: 
\[
\kappa_2=1, \quad \kappa_4=1,  \quad \kappa_6=4 \quad \text{and}\quad \kappa_8=27.
\]
The following fact is a key for understanding the asymptotics of $h(x)$ as $x \to \infty$. We can prove it analogously to \cite[Theorem 1.3 and Proposition A.3]{BG06}. For the reader's convenience a self-contained proof is provided in Appendix \ref{appendix}.  
\begin{lemma} \label{lem:F^{-1}(z)_infty}
For every fixed $N \in \N$ and $\epsilon'\in(0,\pi/4)$ we have 
\begin{equation}  \label{eq:asymptotic_F^{-1}}
\widetilde{F}^{-1}(w) = w + \sum_{n=1}^N \frac{\kappa_{2n}}{w^{2n-1}} + o\left(\frac{1}{w^{2N-1}}\right)  \quad \text{as} \quad w\to\infty, \ w \in  D_{\epsilon'}.  
\end{equation}
\end{lemma}

We provide a proof of Theorem \ref{thm:fine_asymptotics_main} \ref{item:h(x)_infty_main} below. Since the proof requires estimates on $g$ too, we expand the statement of Theorem \ref{thm:fine_asymptotics_main} \ref{item:h(x)_infty_main}  as follows. 

\begin{theorem} \label{thm:fine_asymptotics} The following asymptotic behaviors hold. 

\begin{enumerate}[label=\rm(g${}_\infty$),leftmargin=1.2cm]
\item\label{item:g(x)_infty} 
$\displaystyle 
g(x) = x + \sum_{n=1}^N \frac{\kappa_{2n}}{x^{2n-1}} + o\left(\frac{1}{x^{2N-1}}\right)  \quad \text{as} \quad x\to\infty  
$\quad for every fixed\quad $N \in \N$.  
\end{enumerate}

\begin{enumerate}[label=\rm(h${}_\infty$),leftmargin=1.2cm]
\item \label{item:h(x)_infty_main} 
$\displaystyle h (x) = \frac{1}{e} \sqrt{\frac{\pi}{2}}x^2 e^{-\frac{x^2}{2} } (1 + O(x^{-2}))$\quad  as \quad $x\to\infty$. 
\end{enumerate}
\end{theorem}

\begin{proof} 
\ref{item:g(x)_infty} is just the real part of the formula in Lemma \ref{lem:F^{-1}(z)_infty}. 

For the proof of \ref{item:h(x)_infty_main}, we set $a(x,y):=\im[\widetilde F(x + iy)]$. 
It follows from \eqref{eq:Stieltjes} and Lemma \ref{lem:asymptotic_Cauchy} that, as $x\to \infty$, 
\begin{align}
a(x,0) &= \frac{-\im[\widetilde G(x)]}{|\widetilde G(x)|^2} = \sqrt{\frac{\pi}{2}}x^2e^{-\frac{x^2}{2}} (1 + O(x^{-2})) \quad \text{and}   \label{eq:asymp1} \\
a_y(x,y) &= \im[ i \widetilde F ' (x + iy)] = \re[\widetilde F ' (x + iy)] = 1 + O(x^{-2}),   \label{eq:asymp2}
\end{align}
as long as $x+i y \in D_{\epsilon}$.  
By Taylor's theorem, for every $x >0$ and $ y \in [-1,0]$ there exists $\theta \in (0,1)$ such that $a(x,y) = a(x,0) + a_y (x,\theta y) y $, so that 
\[
a(x,y) =  \sqrt{\frac{\pi}{2}}x^2e^{-\frac{x^2}{2}} (1 + O(x^{-2})) +(1 + O(x^{-2})) y 
\]
If we set the function
\[
f_c(x) := - \sqrt{\frac{\pi}{2}}x^2e^{-\frac{x^2}{2}} (1 + c x^{-2}), \qquad c \in \R, 
\]
then we get 
\[
a(x, f_c(x))  =  \sqrt{\frac{\pi}{2}}x^2e^{-\frac{x^2}{2}} (O(x^{-2}) - c x^{-2}(1+O(x^{-2}))),   
\]
where the $O(x^{-2})$'s are all independent of $c$. 
This implies, for a sufficiently large (fixed) $c>0$, that 
\[
a(x, f_c(x))<0 < a(x,f_{-c}(x)) 
\] 
for all sufficiently large $x$ (such that $f_c(x) >-1 $). 
From the results in Subsection \ref{subsec:boundary}, for $x>0$, the function  $f(x)$ introduced in Theorem \ref{thm:omega} is a unique solution $y \in(-\pi/(2x),0)$ to the equation $a(x,y)=0$. We therefore have $f_c(x) < f(x) < f_{-c}(x)$ for sufficiently large $x >0$, i.e.\  
\begin{equation}\label{eq:asymp_f}
f(x) = - \sqrt{\frac{\pi}{2}}x^2e^{-\frac{x^2}{2}} (1 + O(x^{-2})), \qquad x\to\infty.  
\end{equation}
Finally, using $g(x) = x + \frac1{x} + O(x^{-3})$ we have 
\begin{align}
h(x) = - f(g(x)) 
&= \sqrt{\frac{\pi}{2}}g(x)^2e^{-\frac{g(x)^2}{2}} (1 + O(x^{-2})) \notag  \\
&= \sqrt{\frac{\pi}{2}}x^2 (1+O(x^{-2}))^2 e^{-\frac{x^2}{2}(1+\frac{1}{x^2} + O(x^{-4}))^2} (1 + O(x^{-2})) \notag \\
&=\sqrt{\frac{\pi}{2}}x^2 e^{-\frac{x^2}{2}} e^{-1+ O(x^{-2})} (1 + O(x^{-2}))   \notag \\
&= \frac1{e}\sqrt{\frac{\pi}{2}}x^2 e^{-\frac{x^2}{2}}(1 + O(x^{-2})).   \label{eq:asymp_h}
\end{align}
\end{proof}

\begin{remark}\label{rem:asymptotic}
With some more elaboration, \ref{item:h(x)_infty_main}  can be generalized to higher order expansions of any order
 \begin{equation} \label{eq:precise_h_infty0}
 h (x) = \frac{1}{e} \sqrt{\frac{\pi}{2}}x^2 e^{-\frac{x^2}{2} } \left(1 + \frac{a_2}{x^2} + \frac{a_4}{x^4} + \cdots + \frac{a_{2N}}{x^{2N}} + o(x^{-2N}) \right), \qquad x\to\infty,   
 \end{equation}
where the coefficients $a_{2n}$ are determined by the formula (in the sense of formal power series or asymptotic expansion)
\begin{align}
1 + \sum_{n=1}^\infty \frac{a_{2n}}{x^{2n}}
&= (\widetilde F^{-1})'(x) \exp\left[-\frac{1}{2} \left(\widetilde F^{-1}(x)^2 -x^2 -2\right)\right]   \label{eq:precise_h_infty} \\
&= \left( 1-  \sum_{n=1}^\infty \frac{(2n-1) \kappa_{2n}}{x^{2n}} \right) \exp\left[-\frac{1}{2x^2}  \left(1+ \sum_{n\ge2}\frac{\kappa_{2n}}{x^{2n-2}} \right)^2 -  \sum_{n\ge2}\frac{\kappa_{2n}}{x^{2n-2}}  \right]  \notag \\
&=  1 - \frac{5}{2x^2}  -\frac{43}{8x^4} - \frac{579}{16 x^6}  - \cdots.   \notag
\end{align} 
The proof is sketched below. We first refine \eqref{eq:asymp1} and \eqref{eq:asymp2}. 
From Lemma \ref{lem:asymptotic_Cauchy} we have $\widetilde F(z) = \FN(z) + o(z^{-2N+1})$ and $\widetilde F ' (z) = (\FN)'(z) + o(z^{-2N})$ as $z\to\infty$ with $z \in D_\epsilon$, where 
\[
\FN(x) = x - \sum_{n=1}^N \frac{b_{2n}}{x^{2n-1}} 
\]
for some constants $b_{2n} \in \R ~(n\in \N)$ (the Boolean cumulants of $N(0,1)$). Then we get the following refinement of \eqref{eq:asymp1}: 
\begin{align}
a(x,0) &= \frac{-\im[\widetilde G(x)]}{|\widetilde G(x)|^2} = \sqrt{\frac{\pi}{2}}x^2e^{-\frac{x^2}{2}} \left(\frac{\FN(x)}{x}  + o(x^{-2N})\right)^2, \qquad x\to\infty.   \label{eq:asymp3}
\end{align}
On the other hand, for \eqref{eq:asymp2}, we restrict the variables $(x,y)$ to the thin domain 
\[
J := \{x+ iy: x>0, 2 f_0(x) < y < -2f_0(x) \}. 
\]
The point is that $x+i f(x)$ is contained in $J$ for large $x$ thanks to the established \eqref{eq:asymp_f}. Then we can get the following refinement of  \eqref{eq:asymp2}: 
\begin{equation}
a_y(x,y)  = \re[\widetilde F ' (x + iy)]   = (\FN)'(x) + o(x^{-2N}), \qquad x\to \infty \quad \text{with} \quad x+iy \in J.   \label{eq:asymp4}
\end{equation}
The point is that no $y$'s appear in the main term above thanks to the exponential bounds for $y$.  

Then substituting $y=f(x)$ into $a(x,y) = a(x,0) + a_y (x,\theta y) y $ and combining it with \eqref{eq:asymp3} and \eqref{eq:asymp4} yield
\[
f(x) = - \sqrt{\frac{\pi}{2}}x^2e^{-\frac{x^2}{2}} \left( 1 + \frac{c_2}{x^2} + \frac{c_4}{x^4} + \cdots + \frac{c_{2N}}{x^{2N}} + o(x^{-2N}) \right), 
\]
where $c_2,c_4, \dots$ are determined by the equation (in the sense of asymptotic expansion) 
\[
 1+ \sum_{n=1}^\infty \frac{c_{2n}}{x^{2n}}  =  \frac{ \left( 1 - \sum_{n=1}^\infty \dfrac{b_{2n}}{x^{2n}}\right)^2}{1 + \sum_{n=1}^\infty \dfrac{(2n-1)b_{2n}}{x^{2n}}   },   
\]
the right hand side of which can be written as $\frac{\widetilde F(x)^2}{x^2 \widetilde F'(x)}$.  Finally computing $h(x) = -f(g(x))$ as in \eqref{eq:asymp_h} yields the desired formula  \eqref{eq:precise_h_infty0}. Note that we can write $g(x) = \widetilde F^{-1}(x)$ in the sense of asymptotic expansion, which is a key ingredient for deriving \eqref{eq:precise_h_infty}. 
\end{remark}

Next, we provide a proof of Theorem \ref{thm:fine_asymptotics_main} \ref{item:h(x)_zero_improved}. The proof needs estimates on the functions $f$, $g$ and $gh$, so we expand the statement of Theorem \ref{thm:fine_asymptotics_main} \ref{item:h(x)_zero_improved}.
The results also offer a better understanding of the boundary of $\Omega$; especially they justify that the curve $\{H(x): x>0\}$ approaches $\partial \Xi$ as $x\to0^+$ as observed  in  Figure \ref{fig1}. 

\begin{theorem}  \label{thm:main_infty} Let $0<\eta<1$ be fixed. The following asymptotic behavior holds. 

\begin{enumerate}[label=\rm(f${}_0$),leftmargin=1.2cm] 
\item\label{item:f(x)_zero_improved} 
$\displaystyle 
f(x) = -\frac{\pi}{2x }[1+ O(e^{-\frac{\pi^2 \eta}{8x^2}})] \quad \text{as} \quad x\to0^+. 
$ 
\end{enumerate}
\begin{enumerate}[label=\rm(g${}_0$),leftmargin=1.2cm] 
\item\label{item:g(x)_zero_improved} 
$\displaystyle 
g(x) =\sqrt{-\log \frac1{\sqrt{2\pi}\, x} + \sqrt{ \left(\log \frac1{\sqrt{2\pi}\, x}\right)^2+ \frac{\pi^2}{4} }} + O(x^\eta) \quad \text{as} \quad x\to0^+. 
$ 

In particular, $g(x) \sim \frac{\pi}{\sqrt{8\log \frac1{x}}}$. 
\end{enumerate}
\begin{enumerate}[label=\rm(gh${}_0$),leftmargin=1.2cm] 
\item\label{item:g(x)h(x)_zero} 
$\displaystyle 
g(x)h(x) = \frac{\pi}{2}\left[ 1 + O(x^\eta)\right]  \quad \text{as} \quad x\to0^+. 
$
\end{enumerate}
\begin{enumerate}[label=\rm(h${}_0$),leftmargin=1.2cm] 
\item\label{item:h(x)_zero_improved} 
$\displaystyle 
h(x) = \sqrt{\log \frac1{\sqrt{2\pi}\, x} + \sqrt{ \left(\log \frac1{\sqrt{2\pi}\, x}\right)^2+ \frac{\pi^2}{4} }} + O(x^\eta)  \quad \text{as} \quad x\to0^+. 
$

In particular, $h(x) \sim \sqrt{2\log \frac1{x}}$. 
\end{enumerate}
\end{theorem}

\begin{proof} For the proof it is worth noting that 
\[
\Xi \cap \C^-  \cap \{\re(z) >0\} =  \left\{x+ iy: x >0,  - \frac{\pi}{2x} < y <0 \right\}.   
\] 

\begin{enumerate}[leftmargin=1cm]

\item[\ref{item:f(x)_zero_improved}] Recall from Proposition \ref{prop:curve} that, for $x>0$, $f(x) = -h \circ g^{-1} (x)$ is a unique solution $y \in (-\frac{\pi}{2x},0)$ to $\im [\widetilde{G}(x+iy)]=0$.  We set $\epsilon = \epsilon(x) =e^{-\frac{\pi^2 \eta}{8x^2}}$. Formula \eqref{eq:CauchyND} implies that, for $x>0$, 
\begin{align*}
\im&\left[ \widetilde{G}\left(x-\frac{\pi}{2x} (1-\epsilon) i \right) \right]\\
&= \im\left[G\left(x-\frac{\pi}{2x}(1-\epsilon) i\right) - 2\pi i  \frac{1}{\sqrt{2\pi}} e^{-\frac{1}{2}(x^2-\frac{\pi^2}{4x^2}(1-\epsilon)^2)} e^{\frac{\pi}{2}(1-\epsilon) i} \right] \\
&= \im \left[G\left(x-\frac{\pi}{2x}(1-\epsilon) i\right)\right]-\sqrt{2\pi} e^{-\frac{1}{2}(x^2-\frac{\pi^2}{4x^2}(1-\epsilon)^2)} \sin \frac{\pi\epsilon}{2}
\end{align*}
and therefore, as $x\to 0^+$, 
\begin{align*}
\im \left[\widetilde{G}\left(x-\frac{\pi}{2x}(1-\epsilon) i \right) \right] 
 &= O(x) - (1+o(1))  \frac{\sqrt{2\pi^3}}{2} \epsilon e^{\frac{\pi^2}{8x^2}(1-\epsilon)^2} \\
 &=  o(1) - (1+o(1))  \frac{\sqrt{2\pi^3}}{2} e^{\frac{\pi^2}{8x^2}(1-\eta+o(1))}. 
\end{align*}
In particular, $\im \left[\widetilde{G}\left(x-\frac{\pi}{2x}(1-\epsilon) i \right) \right] <0$ for sufficiently small $x>0$. On the other hand, recall from Lemma \ref{cor:F} that $\im \left[\widetilde{G}\left(x-\frac{\pi}{2x} i \right) \right] >0$. Therefore, $-\frac{\pi}{2x} < f(x) < -\frac{\pi}{2x}(1-\epsilon)$ for sufficiently small $x>0$; in particular
\ref{item:f(x)_zero_improved} holds. 

\item[\ref{item:g(x)_zero_improved}]  We begin with estimating 
\begin{align*}
 \widetilde{G}\left(x+ if(x) \right) 
&= G\left(x+ if(x)\right) - 2\pi i  \frac{1}{\sqrt{2\pi}} e^{-\frac{1}{2}[x^2-f(x)^2]} e^{-i x f(x)}  \\
&= o(1) -i  \sqrt{2\pi}e^{-\frac1{2}x^2}  e^{\frac{\pi^2}{8x^2}[1+ O(e^{-\frac{\pi^2 \eta}{8x^2}})]} e^{\frac{i\pi }{2}[1+O(e^{-\frac{\pi^2 \eta}{8x^2}}) ]}  \\
&= \sqrt{2\pi} (1+O(e^{-\frac{\pi^2 \eta'}{8x^2}}))e^{-\frac1{2}x^2} e^{\frac{\pi^2}{8x^2}},   \qquad x \to0^+, 
\end{align*}
for any  $\eta'  \in (0,\eta)$. This yields 
\begin{equation}\label{eq:H(y)}
y:= \widetilde{F}\left(x+ if(x) \right)  = \frac1{\sqrt{2\pi}} (1+O(e^{-\frac{\pi^2 \eta'}{8x^2}}))e^{\frac1{2}x^2} e^{- \frac{\pi^2}{8x^2}}. 
\end{equation}
Note that $y\to 0^+$ as $x\to0^+$ and $H(y) = x + if(x)$, so that $g(y)=x$. Taking the logarithm of \eqref{eq:H(y)} we obtain $x^4 -2x^2 \log (\sqrt{2\pi}\, y) -\frac{\pi^2}{4}+ O(e^{-\frac{\pi^2 \eta'}{8x^2}})=0$ and hence (by the quadratic formula)
\begin{equation} \label{eq:g(y)^2}
x^2 = \log (\sqrt{2\pi}\, y) \pm \sqrt{ (\log (\sqrt{2\pi}\, y) )^2 
+ \frac{\pi^2}{4} +   O(e^{-\frac{\pi^2 \eta'}{8x^2}})}. 
\end{equation}
The indefinite sign above is actually $+$ because $x\to 0^+$ as $y\to0^+$; then,   
in particular, $x^2 \sim \frac1{8}\cdot \frac{\pi^2}{\log \frac{1}{y}}$. This implies $O(e^{-\frac{\pi^2 \eta'}{8x^2}}) = O(y^{\eta'+o(1)})$ and hence 
\begin{align}
\sqrt{ (\log (\sqrt{2\pi}\, y) )^2 + \frac{\pi^2}{4} +   O(e^{-\frac{\pi^2 \eta'}{8x^2}})} 
&= \sqrt{ (\log (\sqrt{2\pi}\, y) )^2   \notag
+ \frac{\pi^2}{4}} \sqrt{1 +   \frac{O(y^{\eta'+o(1)}) }{(\log(\sqrt{2\pi}\, y) )^2 
+ \frac{\pi^2}{4}  } }   \notag  \\
&=   (1+O(y^{\eta'+o(1)}(\log y)^{-2}))\sqrt{ (\log (\sqrt{2\pi}\, y) )^2 
+ \frac{\pi^2}{4}}   \notag \\
&=  \sqrt{ (\log(\sqrt{2\pi}\, y) )^2 
+ \frac{\pi^2}{4}}  +O(y^{\eta' +o(1)}).   \label{eq:landau}
\end{align}

Combining \eqref{eq:g(y)^2} and \eqref{eq:landau}, taking the square root and similarly handling the Landau symbols yields the desired \ref{item:g(x)_zero_improved}.  (Recall that $0< \eta' < \eta<1$ were arbitrary.)

\item[\ref{item:g(x)h(x)_zero}] Substituting $g(x) \sim \frac{\pi}{\sqrt{8 \log 1/x}} ~(x\to0^+)$ into \ref{item:f(x)_zero_improved},  we get 
\begin{align}
g(x) h(x) &= -g(x) f(g(x)) = \frac{\pi}{2} \left[1+ O\left(e^{-\frac{\pi^2 \eta}{8g(x)^2}   }\right) \right]  = \frac{\pi}{2}\left[ 1 + O(x^{\eta+o(1)})\right]. 
\end{align}
This finishes the proof of \ref{item:g(x)h(x)_zero} since $0<\eta <1$ was arbitrary. 

\item[\ref{item:h(x)_zero_improved}] One only needs to combine \ref{item:g(x)_zero_improved}  and \ref{item:g(x)h(x)_zero}. 

\end{enumerate}
\vspace{-7mm}
\end{proof}

\begin{appendix}

\section{Proofs of basic asymptotic expansions} \label{appendix}

\begin{proof}[Proof of Lemma \ref{lem:asymptotic_Cauchy} (continued)]
By taking the derivatives of \eqref{eq:another_rep} with respect to $z$, similar formulas hold for moments: 
\begin{equation}\label{eq:moments}
\int_{D_\epsilon}  w^n   \frac{1}{\sqrt{2\pi}} e^{-\frac{w^2}{2}}\,\dd w 
= \int_{\R}  x^n   \frac{1}{\sqrt{2\pi}} e^{-\frac{x^2}{2}}\,\dd x =m_n, \qquad n \in \N \cup \{0\}. 
\end{equation}
Since the right hand side of \eqref{eq:another_rep} is analytic in $D_\epsilon$, the identity theorem yields that
\begin{equation*} 
\widetilde{G}(z)=\int_{\partial D_\epsilon} \frac{1}{z-w}\frac{1}{\sqrt{2\pi}} e^{-\frac{w^2}{2}}\,\dd w, \qquad z\in D_\epsilon. 
\end{equation*}

Take $0< \eta <\epsilon< \pi/4$ here. The transform $\widetilde{G}$ obviously has the following representation:
\begin{equation}\label{eq:cauchy2}
\widetilde{G}(z)=\int_{\partial D_{\eta}} \frac{1}{z-w} \cdot \frac{1}{\sqrt{2\pi}} e^{-\frac{w^2}{2}}\,\dd x, \qquad z\in D_\epsilon.
\end{equation}
Combining \eqref{eq:cauchy2} and \eqref{eq:moments} (the latter for $\eta$ instead of $\epsilon$) and the elementary identity
\begin{equation} \label{eq:geometric_series}
\frac{1}{z-w} - \sum_{k=0}^{2N-1} \frac{w^k}{z^{k+1}} = \frac{w^{2N}}{z^{2N}(z-w)}, 
\end{equation}
 we obtain 
\begin{align}
z^{2N+1}\left( \widetilde{G}(z)-\sum_{k=0}^{2N-1} \frac{m_k}{z^{k+1}}\right) 
&=z^{2N+1} \int_{\partial D_{\eta}}\left(   \frac{1}{z-w} -\sum_{k=0}^{2N-1} \frac{w^{k}}{z^{k+1}}\right) \frac{1}{\sqrt{2\pi}} e^{-\frac{w^2}{2}}\,\dd w \notag \\
&= \int_{\partial D_{\eta}} \frac{w^{2N}z}{z-w} \cdot \frac{1}{\sqrt{2\pi}} e^{-\frac{w^2}{2}}\,\dd w, \qquad z \in D_{\epsilon}.\label{eq:Cauchy2}
\end{align}
We here observe that for $w= re^{i(-\frac{\pi}{4}+\eta)} \in \partial D_{\eta}$
\begin{equation}\label{eq:geometric}
\sup_{z\in D_\epsilon}\left| \frac{z}{z-w} \right| \le 1+  \sup_{z\in D_\epsilon}\left|\frac{w}{z-w}\right| = 1+ \frac{1}{\sin (\epsilon - \eta)},
\end{equation}
so that, by the Lebesgue convergence theorem,  \eqref{eq:Cauchy2} converges to the finite number $m_{2N}$ as $z \to \infty$.  This completes the proof of \eqref{eq:asymptotic_Cauchy1}. 

The asymptotic expansion for $\widetilde{G}' (z)$ can be proved very similarly; one needs to use
\[
\widetilde{G}'(z)=-\int_{\partial D_\eta} \frac{1}{(z-w)^2} \cdot \frac{1}{\sqrt{2\pi}}e^{-\frac{w^2}{2}} \,\dd w, \qquad z\in D_\epsilon 
\]
and the $z$-differentiated  version  of formula \eqref{eq:geometric_series}. 
\end{proof}

\begin{proof}[Proof of Lemma \ref{lem:F^{-1}(z)_infty}]
We verify the formula for $N=2$ which should well explain how to handle the general $N$.  Let $0 < \epsilon < \epsilon' < \pi/4$ be fixed. 

A straightforward calculation translates Lemma \ref{lem:asymptotic_Cauchy} into the asymptotic expansion of $\widetilde{F}$ 
\begin{equation}\label{eq:asymp1b}
\widetilde{F}(z) = z - \frac1{z} - \frac{2}{z^3} + O\left( \frac1{z^5} \right), \qquad z \to\infty, \ z \in D_\epsilon. 
\end{equation} 
As the inverse function of $\widetilde{F}(z) = z(1+o(1))$, we obtain $\widetilde{F}^{-1}(w) = w(1+o(1))$. Plugging $\widetilde{F}^{-1}(w)$ into \eqref{eq:asymp1b} yields 
\[
w = \widetilde{F}( \widetilde{F}^{-1}(w)) = \widetilde{F}^{-1}(w)  - \frac1{w(1+o(1))} - \frac{2}{w^3(1+o(1))} + O\left( \frac1{w^5} \right), 
\]
which can be simplified to 
\begin{equation}\label{eq:asymp2b}
\widetilde{F}^{-1}(w) = w + \frac{1}{w} + o\left( \frac1{w} \right). 
\end{equation}
We then plug the refined asymptotics \eqref{eq:asymp2b} into \eqref{eq:asymp1b} to get 
\begin{align}
w &=  \widetilde{F}^{-1}(w)  - \frac1{w + \frac1{w} +o\left(\frac1{w}\right)} - \frac{2}{w^3(1+o(1))} + O\left( \frac1{w^5} \right) \\
&= \widetilde{F}^{-1}(w)  - \frac1{w} - \frac1{w^3} + o\left( \frac1{w^3} \right),  
\end{align}
and hence 
\[
\widetilde{F}^{-1}(w) = w + \frac{1}{w} +  \frac1{w^3}+o\left( \frac1{w^3} \right). 
\]
Higher order expansions can be computed similarly by induction on $N$:  substituting \eqref{eq:asymptotic_F^{-1}} for $N$ into $\widetilde F(z) = z - \sum_{n=1}^{N+1} \frac{b_{2n}}{z^{2n-1}} +o(\frac{1}{z^{2N+1}})$ yields the asymptotic formula \eqref{eq:asymptotic_F^{-1}} for $N+1$. 
From the procedure, the coefficients of $\widetilde F^{-1}$ are obtained as the coefficients of the formal inverse Laurent series of $\widetilde F(z)$. These coefficients are known to be the free cumulants, see e.g.\ \cite[Theorem 12.5]{NS06}. 
\end{proof}

\end{appendix}

\section*{Acknowledgements}
The inserted figures are drawn on Mathematica Version 12.1.1, Wolfram Research, Inc., Champaign, IL. 
This research is supported by JSPS Open Partnership Joint Research Projects Grant Number JPJSBP120209921. Moreover, T. Hasebe was supported by JSPS Grant-in-Aid for Young Scientists 19K14546. Y. Ueda was supported by JSPS Grant-in-Aid for Scientific Research (B) 19H01791 and JSPS Grant-in-Aid for Young Scientists 22K13925.

\vspace{4mm}

\begin{enumerate}
\item[]
\hspace{-9mm}Takahiro Hasebe\\
Department of Mathematics, Hokkaido University.\\
North 10 West 8, Kita-ku, Sapporo 060-0810, Japan.\\
Email address: thasebe@math.sci.hokudai.ac.jp

\item[] 
\hspace{-9mm}Yuki Ueda\\
Department of Mathematics, Hokkaido University of Education. \\
Hokumon-cho 9, Asahikawa, Hokkaido, 070-8621, Japan.\\
Email address: ueda.yuki@a.hokkyodai.ac.jp
\end{enumerate}

\end{document}